\documentclass{article}
\usepackage[a4paper, total={7in, 10in}]{geometry}
\usepackage{amsmath,amsthm,amssymb,mathrsfs,amstext, titlesec,enumitem, comment, graphicx, color, xcolor, inputenc}
\usepackage{hyperref}
\hypersetup{colorlinks,linkcolor={blue},citecolor={blue},urlcolor={red}}
\usepackage{xpatch}
\makeatletter
\def\Ddots{\mathinner{\mkern1mu\raise\p@
\vbox{\kern7\p@\hbox{.}}\mkern2mu
\raise4\p@\hbox{.}\mkern2mu\raise7\p@\hbox{.}\mkern1mu}}
\makeatother

\titleformat{\section}
{\normalfont\Large\bfseries}{\thesection}{1em}{}
\titleformat*{\subsection}{\Large\bfseries}
\titleformat*{\subsubsection}{\large\bfseries}
\titleformat*{\paragraph}{\large\bfseries}
\titleformat*{\subparagraph}{\large\bfseries}

\makeatother
\theoremstyle{plain}
\newtheorem{thm}{Theorem}[section]

\newtheorem{lem}[thm]{Lemma}

\newtheorem{con}[thm]{Conjecture}

\theoremstyle{definition}
\newtheorem{defn}[thm]{Definition}

\newcommand{\thistheoremname}{}
\newtheorem*{genericthm*}{\thistheoremname}
\newenvironment{namedthm*}[1]
  {\renewcommand{\thistheoremname}{#1}%
   \begin{genericthm*}}
  {\end{genericthm*}}

\newcommand{\N}{\mathbb{N}}

\newcommand{\F}{\mathcal{F}}

\newcommand{\bN}{\beta\mathbb{N}}

\newcommand{\FS}{\operatorname{FS}}

\date{\vspace{-5ex}}

\begin{document}

\title{\textbf{Exponential Schur and Hindman Theorem in Ramsey Theory}}
\author{ 
Sayan Goswami\\  \textit{ sayan92m@gmail.com}\footnote{Ramakrishna Mission Vivekananda Educational and Research Institute, Belur Math,
Howrah, West Benagal-711202, India.}
 \and  
	Sourav Kanti Patra\\\textit{souravkantipatra@gmail.com}\footnote{K.S.M. College, Q976+WXP, Aurangabad, Bihar 824101.}}

\maketitle
\begin{abstract}
Answering a conjecture of A. Sisto, J. Sahasrabudhe proved the exponential version of the Schur theorem: for every finite coloring of the naturals, there exists a monochromatic copy of $\{x,y,x^y:x\neq y\},$ which initiates the study of exponential Ramsey theory in arithmetic combinatorics.
%Then using the Stronger Central Sets Theorem, M. Di Nasso and M. Ragosta proved the exponential version of the Hindman theorem.
In this article, We first give two short proofs of the exponential Schur theorem, one using Zorn's lemma and another using $IP_r^\star$ van der Waerden's theorem.  Then using the polynomial van der Waerden theorem iteratively we give a proof of the exponential Hindman theorem. Then applying our results we prove for every natural number $m,n$ the equation $x_n^{x_{n-1}^{\cdot^{\cdot^{\cdot^{x_1}}}}}=y_1\cdots y_m$ is partition regular, which can be considered as the exponential version of a more general version of the  P. Csikv\'{a}ri, K. Gyarmati, and A. S\'{a}rk\"{o}zy conjecture, which was solved by V. Bergelson and N. Hindman independently. As a consequence of our results, we also prove that for every finite partition of $\N,$ there exists two different sequences $\langle x_n\rangle_n$ and $\langle y_n\rangle_n$ such that both the multiplicative and exponential version of Hindman theorem generated by these sequences resp. are monochromatic, whereas in the counterpart in the finitary case, J. Sahasrabudhe proved that both sequences are the same. Our result can be considered as an exponential analog to the result of V. Bergelson and N. Hindman. We also prove that a large class of ultrafilters with certain properties do not exist, which could give us direct proof of the exponential Schur theorem. This result can be thought of as partial evidence of the nonexistence of Galvin-Glazer's proof of the exponential Hindman theorem.
\end{abstract}


\tableofcontents

\section{Introduction}
%For any nonempty set $X$, let $\mathcal{P}_f(X)$ be the set of all nonempty finite subsets of $X.$

Arithmetic Ramsey theory deals with the monochromatic patterns found in any given finite coloring of the
integers or of the natural numbers $\N$. Here ``coloring” means disjoint partition and a set is called ``monochromatic” if it is included in one piece of the partition. Any equation $f(x_1,\cdots ,x_n)=0$ is called partition regular if for every finite coloring of $\N,$ there exists monochromatic $\{x_1,\cdots ,x_n\}$ such that $f(x_1,\cdots ,x_n)=0.$  Arguably the first substantial development in this area of research was due to I. Schur \cite{key-2} in $1916,$ when he proved that the family $\{\{x,y,x+y\}:x\neq y\}$ is a partitioned regular over $\N.$ In other words, we can say that the equation $x+y=z$ is partition regular. For every $n\in \N,$ passing to the induced coloring $n\rightarrow 2^n$ we can say that the equation $x\cdot y=z$ is partition regular. The second cornerstone result in this field of
research is Van der Waerden's Theorem, which states that for any finite coloring of the natural numbers one always finds arbitrarily long monochromatic arithmetic progressions. To state the first substantial development in the direction of infinitary Ramsey theory, we need the following definition.
\begin{defn}[IP set]\text{}
    \begin{enumerate}
        \item For any nonempty set $X$, let $\mathcal{P}_f(X)$ be the set of all nonempty finite subsets of $X.$
        \item For any injective sequence $\langle x_n\rangle_n$ in any commutative semigroup $(S,+)$, a set of the form $$FS(\langle x_n\rangle_n)=\left\lbrace \sum_{t\in H}x_t:H\in \mathcal{P}_f(N)\right\rbrace$$ is called an $IP$ set.
    \end{enumerate}
\end{defn}
Hindman's finite sum theorem  \cite{key-21} is one of the fundamental theorems, and is the first result of infinitary Ramsey theory, which states that 
 for every finite coloring of $\N$, there exists a monochromatic $IP$ set.

Recent research has focused on finding exponential patterns in the Ramsey theory. Answering Sisto's conjecture in \cite{key-4},  J. Sahasrabudhe proved the following theorem.

\begin{thm}[\textbf{Sahasrabudhe-Schur Theorem}]\label{important1}
  For any finite coloring of $\N$, there exists a monochromatic pattern of the form  $\{x,y,x^y\}$. 
\end{thm}
For further development, we refer to the article \cite{js}.
Later using the ultrafilter method, in \cite{key-5}, M. Di Nasso and M. Ragosta proved the Theorem \ref{important1}. Then soon after
  using the Stronger Central Sets Theorem (originally proved by D. De, N. Hindman, and D. Strauss in \cite{DHS}), in \cite{key-6}, M. Di Nasso and M. Ragosta proved a key theorem \cite[Theorem 3.1]{key-6}, and as a consequence, they proved the following exponential version of Hindman's theorem.

\begin{thm}[\textbf{Exponential Hindman Theorem}]\label{important}
For every finite partition of $\N$, there exists an injective sequence $\langle x_n\rangle_{n\in \N}$ such that $$FEP(\langle x_n\rangle_n)=\left\lbrace x_{i_n}^{x_{i_{n-1}}^{\cdot^{\cdot^{\cdot^{x_{i_1}}}}}}:1\leq i_1<\cdots <i_{n-1}<i_n, n\in \N\right\rbrace$$ is monochromatic. We refer to such a pattern as Hindman Tower.
\end{thm}

%In this article, we give two short proofs of Theorem \ref{important1} and a completely different proof of Theorem \ref{important}. Our proof Theorem \ref{important} gives us that strength to classify various ultrafilters witnessing  Theorem \ref{important}. Then as a conclusion we find the exponential version of P. Csikv\'{a}ri, K. Gyarmati, and A. S\'{a}rk\"{o}zy conjecture  \cite{comb}: for every natural number $m,n$ the equation $x_n^{x_{n-1}^{\cdot^{\cdot^{\cdot^{x_1}}}}}=y_1\cdots y_m$ is partition regular.

Before proceeding to the original statements of our works, we need to recall some technical definitions and basic results from the Algebra of the Stone-\v{C}ech compactification. Here we prefer to discuss this theory briefly. The readers familiar with the basics of this theory can skip this subsection and go directly to Section \ref{ourresults}.

%an infinitary extension of Sahasrabudhe's result, mentioned in the abstract. We name such tower-like structures as Hindman Tower. For any sequence $\langle x_n\rangle_n$, denote by 
%$FEP(\langle x_n\rangle_n)=\left\lbrace x_{i_n}^{x_{i_{n-1}}^{\cdot^{\cdot^{\cdot^{x_{i_1}}}}}}:1\leq i_1<\cdots <i_{n-1}<i_n, n\in \N\right\rbrace.$

\subsection{Ultrafilters and its Connections with Ramsey Theory}
Ultrafilters are set-theoretic objects that are intimately related to Ramsey theory. In this section, we give a brief introduction to this theory. For details, we refer to the article  \cite{key-22} to the readers. 
A filter $\mathcal{F}$ over any nonempty set $X$ is a collection of subsets of $X$ such that

\begin{enumerate}
    \item $\emptyset \notin \mathcal{F}$, and $X\in \mathcal{F}$,
    \item $A\in \mathcal{F}$, and $A\subseteq B$ implies $B\in \mathcal{F},$ and
    \item $A,B\in \mathcal{F}$ implies $A\cap B\in \mathcal{F}.$
\end{enumerate}
Using Zorn's lemma we can guarantee the existence of maximal filters which are called ultrafilters. Any ultrafilter $p$ has the following property:
\begin{itemize}
    \item if $X=\bigcup_{i=1}^rA_i$ is any finite partition of $X$, then there exists $i\in \{1,2,\ldots ,r\}$ such that $A_i\in p.$
\end{itemize}

Therefore to prove a certain structure is monochromatic, it is sufficient to study ultrafilters each of whose members contain that structure.

\subsubsection{Space of Ultrafilters}
Let $(S,\cdot)$ be a discrete semigroup. Denote by $\beta S$ the collection of all ultrafilters over $S.$ For any $(\emptyset \neq)A\subseteq S,$ let $\overline{A}=\{p\in \beta S:A\in p\}.$ The collection of sets $\{\overline{A}:(\emptyset \neq)A\subseteq S\}$ forms a basis for a topology. One can show with this topology $\beta S$ is the Stone-\v{C}ech compactification
of $S.$ For any $A\subseteq S,$ and $x\in S$ denote $x^{-1}A=\{y:x\cdot y\in A\}.$ Now one can naturally extend the operation `$\cdot$' over the set $\beta S:$ for any $p,q\in \beta S$, let us define $p\cdot q$ by
$$A\in p\cdot q\iff \{x:x^{-1}A\in q\}\in p.$$
One can show with this definition $(\beta S,\cdot )$ is a semigroup. In fact one can prove that
 $(\beta S,\cdot)$ is a compact right topological semigroup. Here right topological means right action is continuous: for each $q\in \beta S,$ the map $\rho_q:\beta S\rightarrow \beta S$ defined by $\rho_q(p)=p\cdot q$ is continuous. However, the left action is not always continuous but: for $x\in S,$  the map  $\lambda_x:\beta S\rightarrow \beta S$ defined by $\lambda_x(p)=x\cdot p$ is continuous.
 Now from Ellis Theorem, we can say that any closed subsemigroup of $(\beta S,\cdot)$ contains idempotents. However, with some little modification, we can define the notions of $IP$ sets over any arbitrary semigroup, but we are not going for that. One can see that
 \begin{itemize}
     \item If $(S,\cdot )$ is a semigroup, then a set $A\subseteq S$ contains an $IP$ set if and only if there exists an element $p\in E(\beta S,\cdot)$ such that $A\in p.$
\end{itemize}

 Note that if $f:(S,\cdot )\rightarrow (T,\cdot)$ be any map between discrete semigroups. Then $\tilde{f}:\beta S\rightarrow \beta T$ be the continuous extension of $f$ defined by $\tilde{f}(p)=\tilde{f}(\lim_{x\rightarrow p} x).$
 
\subsubsection{Notions of Large Sets}
In the study of ultrafilter theory, one can show some members of $\beta S$ are combinatorial rich. In fact one can characterize combinatorially the members of such ultrafilters. During our study we will apply both combinatorial and algebraic characterizations of those ultrafilters. Here we summarize some of the notions of large sets intimately related to our study. 
%For detailed studies on large sets, we refer to the articles \cite{hindmansize1, hindmansize2}.

\begin{defn}\label{rev2defn} Let $(S,\cdot)$ be a semigroup, let $n\in\N$ and let $A\subseteq S$. We say that
\begin{enumerate} 
\item  for any subset $T\subseteq (\beta S,\cdot)$, we let $E(T)$ be the collection of all idempotents of $T;$ 
\item $A$ is a {\it thick set} if for any finite subset $F\subset S$, there exists an element $x\in S$ such that $Fx=\{fx:f\in F\}\subset A$;
\item $A$ is a {\it syndetic set} if there exists a finite set $F\subset S$ such that $S=\bigcup_{x\in F}x^{-1}A$, where $x^{-1}A=\{y:xy\in A\}$;
\item $A$ is {\it piecewise syndetic set} if there exists a finite set $F\subset S$ such that $\bigcup_{x\in F}x^{-1}A$ is a thick set. It is well known that $A$  is piecewise syndetic if and only if there exists $p\in K(\beta S,\cdot)$ such that $A\in p.$

%\item $A$ is {\it central set} if it belongs to a minimal idempotent in $\beta S$.
\end{enumerate} \end{defn}
A set $L\subseteq \beta S$ is said to be a left ideal if $ \beta S\cdot L=\{q\cdot p:p\in L,q\in \beta S\}\subseteq L.$ Similarly $R\subseteq \beta S$ is said to a left ideal if $ R\cdot \beta S=\{p\cdot q:p\in R,q\in \beta S\}\subseteq R.$ Again using Zorn's lemma one can show that minimal left/ right ideals (w.r.t. inclusion) exist. It can be proved that minimal left ideals are closed.
One can show that if $A\subseteq S $ is a piecewise syndetic set, then there exists a minimal left ideal $L$ of $(\beta S,\cdot)$, and $p\in L$ such that $A\in p.$ Denote by $K(\beta S,\cdot)$, the union of all minimal left ideals of $(\beta S,\cdot)$. It can be proved that $K(\beta S,\cdot)$ is also the union of all minimal right ideals of $(\beta S,\cdot)$. That is $$
\begin{aligned}
	K(\beta S) & =  \bigcup\{L:L\text{ is a minimal left ideal of }\beta S\}\\
	&=  \bigcup\{R:R\text{ is a minimal right ideal of }\beta S\}.
\end{aligned}$$
As each minimal left ideal is the image of $\beta S$ right action map, it is closed, and hence it contains idempotents. Hence, $ E\left(K(\beta S,\cdot)\right)\neq \emptyset $.

%cENTRAL SETS PLAY A VITAL ROLE IN rAMSEY THEORY.
Central sets play a vital role in Ramsey theory.
The notion of central sets was first introduced by H. Furstenberg \cite{F}, using the methods of topological dynamics. Later in \cite{central}, V. Bergelson and N. Hindman first proved an equivalent version of central sets in $\N$, using ultrafilter theory, and then in \cite{SY}, H. Shi and H. Yang extended this result for arbitrary semigroups.% For details, on the history of Central sets, and their combinatorial applications, we refer to the articles \cite{ etds}.
Here we recall the algebraic definition.
\begin{defn}[\textbf{Central set}]
A set $A\subseteq S $ is said to be central if there exists $p\in E\left(K(\beta S,\cdot)\right)$ such that $A\in p.$
\end{defn}

A set $A\subseteq \N$ is said to be an $IP$ set if there exists an injective sequence $\langle x_n\rangle_n,$ such that $A=\FS(\langle x_n\rangle_n),$ and for some $r\in \N$, it is said to be an $IP_r$ set if there exists a sequence $\langle x_n\rangle_{n=1}^r$ such that $A=\FS(\langle x_n\rangle_{n=1}^r)=\{\sum_{t\in H}x_t:(\neq \emptyset )H\subseteq \{1,2,\ldots ,r\}\}.$
If $\F$ is a family of sets, then a set $A$ is dual to this family if $A\cap F\neq \emptyset$ for every $F\in \F$. After we prove a certain structure is preserved in a large set, an immediate question appears: how much structure is preserved in a single cell? In this type of study, we need these notions of dual sets. Some well-known dual sets that we use are $IP^\star$ sets, $IP_r^\star$ sets, etc. Note that if $\F$ and $\mathcal{G}$ are two families of subsets of $S$ satisfying $\F\subseteq \mathcal{G}$, then each $ \mathcal{G}^\star$ set is also a $ \mathcal{F}^\star$. As a consequence, every $IP_r^\star$ set is an $IP^\star$ set.

\subsection{Our Results}\label{ourresults}
\subsubsection{Sahasrabudhe-Schur Theorem}

To capture exponential patterns, in \cite{key-4}, J. Sahasrabudhe used $\log_2$ base partitions. For example, if $\N=\cup_{i=1}^rA_i$ is any finite partition of $\N$, then we can induce another partition of $\N=\cup_{i=1}^rB_i$ by $$m\in B_i\iff 2^m\in A_i.$$ Then if we can show that there exists some $i\in \{1,2,\ldots ,r\}$ such that $\{a,b,a2^b\}\subset B_i$, then this immediately implies $\{x=2^a,y=2^b,x^y=2^{a2^b}\}\subset A_i.$ This immediately reduces the complexity of Theorem \ref{important1} to the problem of finding partition regularity of $\{a,b,a2^b\}.$
%Note that instead of $2$, we can use any positive integer $n>1.$ This simply proves Theorem \ref{important1}. In fact we prove the following stronger version of Theorem \ref{important1}.
In this article, we prove the following theorem which strengthens Theorem \ref{important1} by capturing a large class of ultrafilters each of whose members contains a pattern in Theorem \ref{important1}.

\begin{thm}\label{main1} Let $n\in\mathbb{N}, n>1$ and let $A\subseteq\N$ be multiplicatively piecewise syndetic. Then there exists $x,y\in A$ such that $xn^{y}\in A$. \end{thm}

The importance of the above ultrafilter strengthening will be explored in Section \ref{3}, where we will capture several new exponential combinatorial patterns.

\subsubsection{Exponential Hindman Theorem}

% Taking into account the non-commutativity of exponentiation, this can be seen as an exponential analog of the Hindman theorem. 
 Let $(S,\cdot )$ and $(T,\cdot )$ be two discrete semigroups. Then the tensor product of $p$ and $q$ is defined as
$$p\otimes q=\{A\subseteq S\times T: \{x\in S: \{y:(x,y)\in A\}\in q\}\in p\}$$ 
where $x\in S$, $y\in T.$ For detailed application of tensor pairs on combinatorics see the article \cite{bhw}.
For the function $f:\N^{2}\rightarrow \N$ defined by $f(n,m)=2^{n}m,$ let $\overline{f}:\beta\left(\N^{2}\right)\rightarrow\bN$ be the continuous extension of the function $f$ over $\bN$. Define the operation $``\star"$ as the restriction of $\overline{f}$ to tensor pairs, namely: for all $p,q\in\beta\N$
\[p\star q=\overline{f}(p\otimes q)=2^{p}\cdot q.\]

 %In \cite{key-6}, Di Nasso and Ragosta used a groupoid operation $\star$ over $\bN$ defined as follows: first, consider the operation $f:\N^{2}\rightarrow \N$ defined as $f(n,m)=2^{n}m$. Let $\overline{f}:\beta\left(\N^{2}\right)\rightarrow\bN$ be its continuous extension. $\star$ is the restriction of $\overline{f}$ to tensor pairs, namely: for all $p,q\in\beta\N$
%\[p\star q=\overline{f}(p\otimes q)=2^{p}\cdot q.\]

 %We will discuss $\star$ in detail in section \ref{sour}.
 The following, which is Theorem 3.1 in \cite{key-6} implies the exponential version of the Hindman theorem.

\begin{thm}\label{DiR} 
 Let $p\in E\left(K(\beta\mathbb{N},+)\right)$
and $A\in p\star p$. For every $N\in\mathbb{N}$ and for every sequence
$\varPhi=\left(f_{n}:n\geq2\right)$ of functions $f_{n}:\mathbb{N}\rightarrow\mathbb{N},$
there exists a sequence $\left(a_{k}\right)_{k\in\mathbb{N}}$ such
that
\[
\mathcal{F}_{k,\varPhi}\left(a_{k}\right)_{k\in\mathbb{N}}=\left\{ a_{k}\cdot2^{\sum_{i=1}^{k-1}\lambda_{i}a_{i}}:k\in\mathbb{N},0\leq\lambda_{1}\leq N,\text{ and }0\leq\lambda_{i}\leq f_{i}\left(a_{i-1}\right)\text{ for }2\leq i\leq k-1\right\} \subset A.
\]
\end{thm}

 Let $N=1$, and choose the function $f_2(x)=2^x.$ Inductively assume that we have defined the function $f_n(x)$ for some $n\in \N$. Now define $f_{n+1}(x)=(2^x)^{f_n(x)}$.  Now choosing $\varPhi=\left(f_{n}:n\geq2\right)$, Theorem \ref{DiR} immediately implies Theorem \ref{important}, for details see \cite{key-6}.
Our next result is aimed to find out the polynomial extension of Theorem \ref{DiR}. %As a byproduct, this will allow us to prove that a variation of \cite[Question 25]{key-4} is true. 
Before we state our result, let us recall the polynomial van der Waerden's theorem was proved in \cite{70}. We will apply this theorem iteratively to prove our result.
%The polynomial extension of the van der Waerden's theorem was proved in \cite{70}.
%To prove the polynomial extension of Di Nasso-Ragosta Theorem we  use the recurrence property of the polynomial van der Waerden's theorem.
The following version of the polynomial van der Waerden's theorem is a direct consequence of the polynomial Hales-Jewett theorem, proved in \cite{71}.

\begin{thm}\label{n}
 Let $P$ be a finite collection
of polynomials with integer coefficient, and having no constant term, and let $A$ be an additively piecewise
syndetic set. Then 
\[
\left\{ d\in\N\mid\text{ there exists }a\in\N\text{ such that }\left\{ a,a+p(d):p\in P\right\} \subset A\right\} 
\]
is an $IP_r^{\star}$ set for some $r\in \N$. 
\end{thm}

%For any set $X,$ let $\mathcal{P}_f(X)$ be the collection of all nonempty finite subsets of $X.$ Let $\mathbb{P}$ be the collection of all polynomials with integer coefficients with no constant term. 
Let $\mathbb{P}$ be the set of all polynomials with rational coefficients without constant terms. For every $P\in \mathbb{P},$ $deg(P)$ is the degree of the polynomial $P,$ and $coef(P)=\max \{|c|: c\text{ is a coefficient of }P\}.$ For every $n(>1)\in \N,$ define the operation $\star_n$ on $\N$ by $a\star_nb=n^ab.$ 
The above result is all we need to prove the following polynomial extension of Theorem \ref{DiR}.
%version of the Di Nasso-Ragosta Theorem.

\begin{thm}[\textbf{Polynomial extension of Theorem \ref{DiR}}]\label{main2} Let $n\in \N$, $p\in E\left(K(\beta\mathbb{N},+)\right)$
and $q\in \bN$ be such that for every $N\in \N$, each element of $q$ contains an $IP_N$ sets. Let $n\in \N_{>1},$ and  $A\in p\star_n q$ and  $F_{1}\in \mathcal{P}_f(\mathbb{P})$. Let $\varPhi=\left(f_{n}:n\geq2\right)$ be a sequence
of functions $f_{n}:\mathbb{N}\rightarrow\mathbb{N}.$ 
For every $n(>1),x\in \N$, let $F_{n,x}=\{P\in \mathbb{P}:deg(P)\leq f_n(x)\text{ and } coef(P)\leq f_n(x)\}.$ 
%\textcolor{blue}{I do not understand what you mean here. Rewrite it in a comprehensible form.}

Under this hypothesis we have a sequence $\left(x_{k}\right)_{k\in\mathbb{N}}$
such that

\[
\mathcal{PF}^n_{k,\varPhi}\left(x_{k}\right)_{k\in\mathbb{N}}=\left\{ x_{k}\cdot n^{\sum_{i=1}^{k-1}p_{i}\left(x_{i}\right)}:k\in\mathbb{N},p_{1}\in F_{1},\text{ and }p_{i}\in F_{i,f_{i}\left(x_{i-1}\right)}\text{ for }2\leq i\leq k-1\right\} \subset A.
\]
\end{thm}
%However, later in section $3$, we prove that every multiplicative piecewise syndetic set satisfies the conclusion of Theorem \ref{main2}.

\subsubsection{Non-Existence of Galvin-Glazer Type Proofs}

The ultrafilter proof of the Hindman theorem is attributed to an unpublished paper by Galvin-Glazer, where they observed that an idempotent ultrafilter exists, and this immediately implies the Hindman theorem. But we doubt if such proof exists in the exponential world. This is our study's next goal, where we address whether $``\star_n"$ idempotent ultrafilter exists in $\bN$ or not. This question looks promising as if there exists $p\in \bN$ such that $p\star_np=p$, then it is a routine exercise to check for every $A\in p,$ $A$ contains patterns of the form $\{x,y,xn^y\}.$ Then the ultrafilter $n^p$ will witness the Sahasrabudhe-Schur theorem. The existence of such ultrafilters would be the first breakthrough toward the exponentiations of Galvin-Glazer's proof of the Hindman theorem. Unfortunately, our observation suggests that such an ultrafilter may not exist. So we make the following conjecture.

\begin{con}\label{con2} For any $n\in \N_{>1},$ $\star_{n}$-idempotent ultrafilters do not exist. \end{con}

However, we could not solve this question in its full strength but we proved a relatively weak version of our conjecture.
\begin{thm}\label{5}
For any $n\in \N_{>1}$, there is no  idempotent $p$ of $(\beta \mathbb{N}, \star_{n})$ such that $p\in K(\beta \mathbb{N}, +)\cup E(\beta \mathbb{N}, +)$.
\end{thm}

\subsubsection{Further Combinatorial Applications}

After accumulating all our results, we prove the exponential version of two well-known results in the literature in this section. As we are now losing our temper to start our proofs, we will discuss all these results with their proofs in Section \ref{3}.

%In this section finally, we prove the exponential version of P. Csikv\'{a}ri, K. Gyarmati, and A. S\'{a}rk\"{o}zy conjecture \cite{comb}, which was later proved  and a result of N. Hindman \cite{key-51} by using combinatorial arguments and later in \cite{key-51} V. Bergelson and N. Hindman by using ultrafilter argument. We will discuss all of these results in Section \ref{3}.

\subsection{Structure and Machinery of the Paper:} 
\begin{enumerate}
\item In Section \ref{21}, we first provide two short proofs of Theorem \ref{main1}. In the first proof, we use a variant of van der Waerden's theorem: this proof is purely combinatorial. Then in the second proof, we use the fact that every piecewise syndetic set and its infinitely many translated copies is contained in the same ultrafilter: this proof uses the structure of $\bN.$ Then applying the polynomial van der Waerden theorem we prove Theorem \ref{main2}.

The fundamental idea of Section \ref{21} is an iterative application of the polynomial van der Waerden theorem along with the machinery of the structure of the $\bN.$ When we apply it once, it gives us the Sahasrabudhe-Schur pattern, and multiple times applications give us an infinitary Sahasrabudhe-Schur pattern, which we recognize as the Exponential Hindman theorem.

\item After that, in Section \ref{sour}, we address the exponential version of Galvin-Glazer's proof of the Hindman theorem. We prove for a large class of ultrafilters $``\star_n"$ idempotent ultrafilters do not exist, which implies the exponential Schur theorem. Of course, this is a partial answer to the Conjecture \ref{con2}, whereas the full answer is not known. The proof uses uniform compactification of topological spaces. We thank Prof. D. Strauss for her immense help with this section. This section is more technical than the other sections. The philosophy of the proof is that for a large class of ultrafilters, the existence of $``\star_n"$ idempotence gives rise to the existence of unsolvable equations over the uniform compactification of $([0,\infty),+).$
%This is the only section of this article, where one needs expertise on the theory of compactification.

\item Then in Section \ref{3}, we focus on the study of partition regularity of the equation exponential equal product, and monochromatic product-exponential Hindman patterns.  The first one is an analogous study in the direction of P. Csikv\'{a}ri, K. Gyarmati, and A. S\'{a}rk\"{o}zy conjecture \cite{comb}: the equation $x+y=z$ is partition regular, which was first solved by V. Bergelson and N. Hindman independently in \cite{pissa, h2}. And the second one is the exponential version of the monochromatic pattern of the form $FS\left(\langle x_n\rangle_n\right)\cup FP\left(\langle y_n\rangle_n\right)$ proved in \cite{key-51, h1}.

The proof uses some algebraic arguments of the structure of $\bN,$ some results from the Topological dynamics, proved by N. Hindman, D. Strauss, and L. Zamboni in \cite{hsz}. Basically, we observed that for $s\in \N,$ and the map $f_s:\N\rightarrow \N$ given by $f_s(t)=n^st,$ the system $(\bN,\langle T_s\rangle_{s\in \N})$ forms a topological dynamical system, where $T_s:\bN\rightarrow \bN$ is the continuous extension of $f_s.$ This observation is a bridge between the exponential Ramsey theory and the Topological dynamics.
We strongly believe this observation should have a deep link with the new techniques developed in the recent breakthrough paper of D. Glasscock, and A. N. Le \cite{GL}, which must be explored. 
As many dynamically large sets can be captured using the analysis of filter algebra (developed in \cite{GL}), our dynamical system may be applied to strengthen all the results in this direction, which is to be studied.
%Specially we believe that various Ramsey theoretic large

\item Finally we conclude our paper with Section \ref{count} by constructing an additive-multiplicative thick set that does not contain the Sahasrabudhe-Schur pattern.

%For each $s\in \N,$ let $T_s:\bN\rightarrow \bN$ be the continuous extension of the map $f_s:\N\rightarrow \N$ given by $f_s(t)=n^st$. Clearly $(\bN,\langle T_s\rangle_{s\in \N})$ is a topological dynamical system.
\end{enumerate}

\section{Proof of Exponential Schur and Hindman Theorem}\label{21}

In this section, we prove Theorem \ref{main1}, and Theorem \ref{main2}. We have already discussed all the preliminaries that we need to prove these two theorems. First, we prove Theorem \ref{main1}. The first proof of  Theorem \ref{main1} does not use Zorn's lemma.
We first prove the following theorem by using  $IP_r^\star$ van der Waerden's theorem (this is the linear version of Theorem \ref{n}). Previously we had a proof using the Hales-Jewett theorem, but later an anonymous person helped us with the following proof using the $IP_r^\star$ van der Waerden's theorem. We are grateful to his/her kindness.

\begin{proof}[First proof of Theorem \ref{main1}]
As $A$ is multiplicatively piecewise syndetic, there is a finite set $F$ such $\bigcup_{f\in F}f^{-1}A=T,$ a thick set. Let $k=|F|,$  and using Theorem \ref{n} the fact that multiplicatively piecewise
syndetic sets are $IP_r$ sets for all $r\in \N$, fix $M\in \N$ large enough so that any $k$-coloring of $[M]$ contains a monochromatic set $\{a,a+y\}$ with $y\in A.$ Using the fact
that $T$ is thick, let $t$ be such that $\{tn^m:m\leq M\}\subset T.$ Now k-color
$[M]$ by labeling $m\in M$ by the smallest $f\in F$ with $ftn^m\in A.$ By our
choice of M, there must be $a\in \N$ and $y\in A$ with $x=ftn^a\in A$ and $xn^y=ftn^an^y=ftn^{a+y}\in A$ completing the proof.
    
\end{proof}

The following second proof of Theorem \ref{main1} uses Zorn's lemma: the fact that ultrafilters exist. The elegance of the following proof is that we used only the pigeonhole principle.
%This indigenous construction shows that how we can mix several variables to find our desi

\begin{proof}[Second proof of Theorem \ref{main1}] As $A$ is a multiplicative piecewise syndetic set, there exists $p\in K(\bN,\odot)$ such that $A\in p$. Hence, by \cite[Theorem 4.39]{key-22} the set $S=\{x:x^{-1}A\in p\}$ is syndetic. Let $\N=\bigcup_{t\in F}t^{-1}S$, and $|F|=r.$ As $A$ is a piecewise syndetic  set, it  contains an $IP_{r+1}$, say $FS(\langle b_i\rangle_{i=1}^{r+1})\subset A.$ 

    As $\{n^{b_1},n^{b_1+b_2},\ldots ,n^{b_1+\cdots +b_{r+1}}\}\subset \bigcup_{t\in F}t^{-1}S$ and $|F|=r$, by Pigeonhole principle there exists $t\in F$ and $m,n\in [1,r+1]$ such that $\{tn^{b_1+\cdots +b_m},tn^{b_1+\cdots b_n}\}\subset S.$ Let $m<n$, $b=n^{b_1+\cdots +b_m}$ and $c=b_{m+1}+\cdots b_n$. So, $\{tb,tbn^c\}\subset S$. Choose $a\in A\cap (tb)^{-1}A\cap (tbn^c)^{-1}A\in p.$ If $x=atb\in A$ and $y=c\in A$, by construction $xn^y\in A$.\end{proof}

Now we prove  Theorem \ref{main2}.

\begin{proof}[Proof of Theorem \ref{main2}]
To avoid the complexity in the calculation we will prove this theorem up to $k=3.$ 

The rest of the part can be proved iteratively or inductively. The technique is verbatim.

For any $p\in E\left(\beta\mathbb{N},+\right)$ and any $A\in p,$ define $A^{\star}=\{n\in A:-n+A\in p\}.$ As $p$ is an idempotent, we have  $A^{\star}\in p$.
Let $p\in E\left(K(\beta\mathbb{N},+)\right)$ and $q\in \bN$ be such that for every $N\in \N$, each element of $q$ contains an $IP_N$ sets.

Choose 
$A\in p\star_n q=n^p\cdot q.$ Then $B=\left\{ m_0:\left\{ n_0:n^{m_0}n_0\in A\right\} \in q\right\} \in p.$
Let $m_{1}\in B^{\star}$ and let $C_{1}=\left\{ n':n^{m_1}n'\in A\right\} \in q.$
Let $F_{1}$ be the finite collection of polynomials. Then from Theorem
\ref{n}, the set
\[
D_{1}=\left\{ d:\text{ there exists }a_{1}\text{ such that }\left\{ a_{1},a_{1}+p_{1}\left(d\right):p_1\in F_1\right\} \subset B^{\star}\right\} 
\]
 is an $IP_r^{\ast}$ set for some $r\in \N$. It is straightforward to check that for any $IP_r^\star$ set $C$ and $n\in \N,$ $n^{-1}C$ is an  $IP_r^\star$ set. Hence $\frac{1}{n^{m_{1}}}D_{1}$ is an  $IP_r^\star$ set. Again from \cite[Proposition 2.5]{star} we know that intersection of two $IP_r^{\star}$ set is again a $IP_l^{\star}$ set
 for some $l\in \N.$
 Hence we have  $\frac{1}{n^{m_{1}}}D_{1}\cap D_{1}$ is an $IP_s^{\star}$ set for some $s\in \N$, and so $\frac{1}{n^{m_{1}}}D_{1}\cap D_{1}\in q.$
 So we can choose an element $$n_{1}\in\frac{1}{n^{m_{1}}}D_{1}\cap D_{1}\cap C_{1}\in q.$$
Define $x_{1}=n^{m_{1}}n_{1}\in A.$ By our choice $x_{1}\in D,$ there
exists $a_{1}\in B^{\star}$ such that $$\left\{ a_{1},a_{1}+p_{1}\left(x_{1}\right):p_{1}\in F_{1}\right\} \subset B^{\star}.$$

Define $$B_{2}=B^{\star}\cap\left(-a_{1}+B^{\star}\right)\cap\bigcap_{p_{1}\in F_{1}}-\left(a_{1}+p_{1}\left(x_{1}\right)\right)+B^{\star}\in p;$$

and $$C_{2}=C_{1}\cap\bigcap_{p_{1}\in F_{1}}\left\{ n':n^{a_{1}+p_{1}\left(x_{1}\right)}n'\in A\right\} \cap\left\{ n':n^{a_{1}}n'\in A\right\} \in q.$$
As $B_2^\star$ is again a piecewise syndetic set, we have from Theorem
\ref{n}, the set 
\[
D_{2}= \left\{ d:\text{ there exists }a_{2}\text{ such that }\left\{ a_{2},a_{2}+p_{2}\left(d\right):p_{2}\in F_{2,f_{2}\left(x_{1}\right)}\right\} \subset B_{2}^{\star}\right\} 
\]
 is an $IP_r^{\star}$ set for some $r\in \N$. Now from \cite[Proposition 2.5]{star}, we have $\frac{1}{n^{a_{1}}}D_{2}\cap D_{2}\cap D_{1}\in q.$
So we can choose an element $$n_{2}\in\frac{1}{n^{a_{1}}}D_{2}\cap D_{2}\cap D_{1}\cap C_{2}\in q.$$
Define $n^{a_{1}}n_{2}=x_{2}\in A$ and so  for all $p_{1}\in F_{1}$ we have
\[
n^{a_{1}+p_{1}\left(x_{1}\right)}n_{2}=x_{2}n^{p_{1}\left(x_{1}\right)}\in A.
\]

 So the first step of induction is complete. 
 Now we show how the second step is coming from the first step, and this process can be adapted to show how the $(n+1)^\text{th}$ process comes from the $n^\text{th}$ step.

Now from the choice of $x_{2},$ we have $a_{2}\in B_{2}^{\star}$
such that $$\left\{ a_{2},a_{2}+p_{2}\left(d\right):p_{2}\in F_{2,f_{2}\left(x_{1}\right)}\right\} \subset B_{2}^{\star}.$$
Letting $c=a_{1}+a_{2}$, we have
\[
\left\{ c+p_{1}\left(x_{1}\right)+p_{2}\left(x_{2}\right):p_{1}\in F_{1},p_{2}\in F_{2,f_{2}\left(x_{1}\right)}\right\} \subset B^{\star}.
\]

Define $$C_{3}=C_{1}\cap\bigcap_{p_{1}\in F_{1}}\bigcap_{p_{2}\in F_{2,f_{2}\left(x_{1}\right)}}\left\{ n':n^{c+p_{1}\left(x_{1}\right)+p_{2}\left(x_{2}\right)}n'\in A\right\} \in q$$
and $$B_{3}^{\star}=B^{\star}\cap\bigcap_{p_{1}\in F_{1}}\bigcap_{p_{2}\in F_{2,f_{2}\left(x_{1}\right)}}-(c+p_{1}\left(x_{1}\right)+p_{2}\left(x_{2}\right))+B^{\star}\in p.$$
Again as $B_3^\star$ is a piecewise syndetic set, we have from Theorem \ref{n}, the set
\[
D_{3}= \left\{ d:\text{ there exists }a_{3}\text{ such that }\left\{ a,a+p_{3}\left(d\right):p_{3}\in F_{3,f_{3}\left(x_{2}\right)}\right\} \subset B_{3}^{\star}\right\} 
\]
 is an $IP_r^{\star}$ set for some $r\in \N.$  Now from \cite[Proposition 2.5]{star}, we have $$D_{3}\cap\bigcap_{p_{1}\in F_{1}}\bigcap_{p_{2}\in F_{2,f_{2}\left(x_{1}\right)}}\frac{1}{n^{c+p_{1}\left(x_{1}\right)+p_{2}\left(x_{2}\right)}}D_{3}\in q.$$
 Hence we can choose $$n_{3}\in D_{3}\cap\bigcap_{p_{1}\in F_{1}}\bigcap_{p_{2}\in F_{2,f_{2}\left(x_{1}\right)}}\frac{1}{n^{c+p_{1}\left(x_{1}\right)+p_{2}\left(x_{2}\right)}}D_{3}\cap C_{3}.$$
Define $n^{c}n_{3}=x_{3}\in A,$ and so  for all $p_{1}\in F_{1},\text{ and }p_{2}\in F_{2,f_{2}\left(x_{1}\right)},$ we can choose
\[
n^{c+p_{1}\left(x_{1}\right)+p_{2}\left(x_{2}\right)}n_{3}=x_{3}n^{p_{1}\left(x_{1}\right)+p_{2}\left(x_{2}\right)}\in A.
\]

%Now iterating this argument, we have an infinite sequence $\langle x_{n}\rangle_n$ satisfying the conclusion of our given theorem.
This is how the third step comes from the second step, and now one can follow the induction process similarly to conclude our results.
\end{proof}

\section{On the Non-Existence of Galvin-Glazer Type Proofs of Exponential Hindman Theorem}\label{sour}
As we stated before, in this section, we will prove that for certain ultrafilters, $``\star_n"$ idempotents do not exist.
%The material of this Section, particularly the proof of Theorem \ref{5}, has been influenced by discussions with Prof. Dona Strauss. We are thankful for her help.
%\textcolor{blue}{In what follows, I corrected the english and I tried to follow the mathematics, but I am not an expert in dynamics. What is not clear to me is exactly what was known already and what not and, of what was not known, what has been done by Dona and what by Sourav. Can we be a little more specific? And, again, can you write down the dynamics of $E_{1}, E_{2}$, which should be trivial (I see no reason why $E_{1}(p,q)$ should be different from $\lim_{m\rightarrow q}\lim_{n\rightarrow p} n^{m}$ and analogous for $E_{2}$)? If it is not trivial, discuss why.}
%\textcolor{red}{You are right, this will overlap.}
Before we proceed to our main theorem, we will prove three lemmas. Note that if $(X,\mathcal{U})$ is a uniform space, then there is a topological compactification $(\phi, \gamma_uX)$ of $X$ such that it is precisely the uniformly continuous functions in $C_\mathbb{R}(X)$ which have continuous extensions to $\gamma_uX.$ This $\gamma_uX$ will be called the uniform compactification of $X$. For details, we refer to \cite[Chapter 21]{key-22}.
\begin{lem}\label{2}
 Let $x\in [1, \infty)$ and $q\in U$, where $U$ is the uniform compactification of $([0, \infty), +)$. Then the equation $x+q=q$ can not hold.
\end{lem}
\begin{proof}
 Choose an integer $n>x$. Now the map $\varphi:[0, \infty)\longrightarrow\mathbb{T}$ given by $\varphi(t)=exp \frac{2\pi it}{n}$ is a uniformly continuous homomorphism, where $\mathbb{T}=\{z\in \mathbb{C}: \Vert z \Vert=1\}$ is the compact circle group. So it extends to a continuous homomorphism $\widetilde{\varphi}:U\longrightarrow\mathbb{T}$. Now if $x+q=q$, then $\varphi(x)=1$, a contradiction.
\end{proof}
If $q$ is an idempotent in $(\beta \mathbb{N}, \star_{n})$ for some $n\in \N$, then $q=n^{q}\cdot q$. This implies that $q\in \mathbb{H}_n$, where $\mathbb{H}_n=\bigcap_{m\in \mathbb{N}}\overline{n^m\mathbb{N}}$ is a subsemigroup of $(\beta \mathbb{N}, +)$.

\begin{lem}\label{wow}
 If $p$ is an idempotent in  $(\beta \mathbb{N}, \star_{n})$ for some $n\in \N_{>1}$, then there is no $u\in \beta \mathbb{N}$ for which $p=u+ p$. In addition, for any $p,q \text{ and }u\in \bN$, the equation $q\star_{n} p=p=u+p$ is not solvable.
\end{lem}
\begin{proof}
 Let $U$ be the uniform compactification of $([0, \infty), +)$ and $f:\beta \mathbb{N}\longrightarrow U$ be the continuous extension to $\beta\N$ of the function $\log_{n}:\N\rightarrow [0,\infty)$. $f$ has the property that $f(x\cdot y)=f(x)+f(y)$ for every $x, y\in \beta \mathbb{N}$ and $f(x+y)=f(y)$ for every $x, y\in \beta \mathbb{N}$ and for every $y\in\mathbb{N^*}=\beta \mathbb{N}\setminus \mathbb{N}$.
 By contrast, assume that $p=q\star_{n} p=u+p$ for some $u\in\bN$.
% \textcolor{blue}{Last line makes no sense. I guess you are trying to say something about the product, and where is $y$ from exactly? Also, do not use $\star$ to denote other things here or it gets messy}\textcolor{red}{done '*'}

 {\bfseries Claim:} There exists at most one $m\in\N$ such that $m+ p\in \mathbb{H}_n$. 
 
 To prove the claim, let $a+ p, b+ p\in \mathbb{H}_n$ for distinct $a, b\in \mathbb{N}$ and choose $m\in \mathbb{N}$ such that $n^m>\max\{a, b\}$. Let $h:\beta \mathbb{N}\longrightarrow \mathbb{Z}_{n^m}$ denote the extension to $\bN$ of the natural homomorphism that maps $\mathbb{N}$ onto $\mathbb{Z}_{n^m}$. Then $h(a)+h(p)=h(b)+h(p)$. So $h(a)=h(b)$, which implies that $n^m|a-b$, a contradiction.
  Having proven the Claim, now observe that $p\in cl\{n^m\cdot p\mid  m\in \mathbb{N}\}\cap cl\{m+ p\mid  m\in \mathbb{N}, m+p\notin \mathbb{H}_n\}$, because $p=n^q\cdot p=u+p$ for some $q,u\in \bN$. Hence from \cite[Corollary 3.42]{key-22}, 
  %\textcolor{blue}{I do not think this applies here} \textcolor{red}{ It is ok. Think contrapositive statement}
 %\textcolor{blue}{What is this citation supposed to be?}\textcolor{green}{done}
 there exists $x\in \beta \mathbb{N}$ and $m\in \mathbb{N}$ with $m+p\notin \mathbb{H}_n$, such that $n^{x}\cdot p=m+p$ or, $n^{m}\cdot p=x+p$. The first possibility cannot hold because $m+p\notin \mathbb{H}_n$ but $n^{x}\cdot p\in \mathbb{H}_n$. The second possibility also cannot hold because it implies that $m+f(p)=f(p)$,  which makes a contradiction in Lemma \ref{2}.
 
% \textcolor{blue}{What is P?}\textcolor{green}{this will be small $p$. Done.}

\end{proof}
The following is a basic observation that we prove in detail for completeness.

\begin{lem}\label{4}
 If $p\in K(\beta \mathbb{N}, +)$, then $u+p=p$ for some $u\in K(\beta \mathbb{N}, +)$.
\end{lem}
\begin{proof}
 Choose a minimal left ideal $L$ of $(\beta \mathbb{N}, +)$ such that $p\in L$. Then $L\subseteq  K(\beta \mathbb{N}, +)$. Now $L+ p$ is also a left ideal of $(\beta \mathbb{N}, + )$ and $L+ p\subseteq L$. By the minimality of $L$, we have that $L+ p=L$. Thus $p\in L+ p$. Take $u\in L\subseteq  K(\beta \mathbb{N}, +)$ such that $p=u+ p$, as required.
\end{proof}

The three Lemmas above allow to deduce that $\star_{n}$-idempotents, if they exist, cannot belong to $K\left(\beta\N,+\right)$.

%\textcolor{blue}{what about $\overline{K\left(\beta\N,\oplus\right)}$?}\textcolor{green}{Don't know.}

\begin{proof}[Proof of Theorem \ref{5}]
 If $p\in K(\beta \mathbb{N}, +)$, then there exists $u\in K(\beta \mathbb{N}, +)$ such that  $u+ p=p$ by Lemma \ref{4}. Now by Lemma \ref{wow}, $p$ can not be an idempotent of  $(\beta \mathbb{N}, \star_{n})$.

  If $p\in E(\beta \mathbb{N}, +)$, then $p+ p=p.$ Now by Lemma \ref{wow}, $p$ can not be an idempotent of  $(\beta \mathbb{N}, \star_{n})$.
\end{proof}

Putting together the results of this section, we have the limitation that $\star_{n}$-idempotents (for any $n\in \N_{>1}$), if they exist, must belong to $\mathbb{H}_n\setminus \left(K(\beta \mathbb{N}, +)\cup E(\beta \mathbb{N}, +)\right)$.

%\section{Combinatorial applications}

\section{Further Combinatorial Applications}\label{3}

\subsection{Monochromatic Solutions of Exponential Equations}

In \cite{comb}, P. Csikv\'{a}ri, K. Gyarmati, and A. S\'{a}rk\"{o}zy conjectured that the equation $a+b=c\cdot d$ is partition regular. In \cite{pissa, h2}, V. Bergelson and N. Hindman independently answered this conjecture. Here we address the analog version of this result in an exponential setting. We show how Theorem \ref{main1} implies the exponential version of this theorem.

\begin{thm}\label{trivial}
     The equation $xy=a^{b}$ is partition regular. 
     More generally, every multiplicative piecewise syndetic set contains solutions of this equation.
\end{thm}
     
\begin{proof}
Choose any $p\in K(\beta\N,\cdot)$ and $A\in p$. Then by \cite[Theorem 4.39]{key-22}, the set $S=\{x\in \N:x^{-1}A\in p\}$ is syndetic. Hence there exists $a_1,b_1\in \N$ such that $\{a_1,b_1,a_1+b_1\}\subset S.$ Hence $B=A\cap \bigcap_{x\in \{a_1,b_1,a_1+b_1\}}x^{-1}A\in p.$ Now from Theorem \ref{main1}, there exist $\{c,d,c2^d\}\subset B.$ Hence $(a_1+b_1)c2^d=a_1c2^d+b_1c2^d\in A.$ 
Define
\begin{itemize}
    \item $a=2^{(a+b)\cdot c};$
    \item $b=2^d;$
    \item $x=2^{a_1c2^d}$ and $y=2^{b_1c2^d}.$
\end{itemize}
Then clearly $\{a,b,x,y\}\subset A,$ and $xy=a^{b}.$
\end{proof}

%In the statement of Theorem \ref{trivial} we decided to highlight the equation $xy=a^{b}$ as it resembles an exponential analog of the equation $x+y=ab$, whose partition regularity has been an open problem for a few years
An immediate question appears if there exists a finitary extension of Theorem \ref{trivial}. In \cite{partition, h2}, it has been proved that for every $m,n\in \N,$ the equation $x_1+\cdots +x_n=y_1\cdots y_m$ is partition regular. Now we prove the following theorem which is a finitary extension of Theorem \ref{trivial}.
\begin{thm}\label{final}
    For every $m,n\in \N,$ the equation $x_n^{x_{n-1}^{\cdot^{\cdot^{\cdot^{x_1}}}}}=y_1\cdots y_m$ is partition regular. More generally, every multiplicative piecewise syndetic set contains solutions of this equation.
\end{thm}
The proof goes through some technical lemmas, where first we prove that every multiplicative piecewise syndetic set witness Theorem \ref{main2}, then proceeding similarly to the proof of Theorem \ref{trivial} we prove our result.

%In this section, we prove that every piecewise syndetic set satisfies the conclusion of Theorem \ref{main2}.
Let $\mathcal{DR}$ be the collection of those ultrafilters $p$ such that each member of $p$ satisfies the conclusion of theorem \ref{main2}, i.e. $$\mathcal{DR}=\{p:p\text{ satisfies conclusion of theorem }\ref{main2}\}.$$
The following lemma shows that the set $\mathcal{DR}$ is a left ideal of $(\bN,\cdot).$
\begin{lem}\label{dilation}
    Let $p\in \mathcal{DR}$. Then for every $q\in \bN$, we have $q\cdot p\in \mathcal{DR}.$
    
    %sequence $\varPhi=\left(f_{n}:n\geq2\right)$
%of functions $f_{n}:\mathbb{N}\rightarrow\mathbb{N},$ let $F_{n,x}$
%be a finite collection of polynomials defined as in Theorem \ref{poly}. Then, if $A\in q\odot (p\star_n p)$, there exists a sequence $\langle a_n\rangle_n$ such that $\mathcal{PF}^n_{k,\varPhi}\left(a_{k}\right)_{k\in\mathbb{N}}\subset A.$
\end{lem}
\begin{proof}
    If $A\in q\cdot p$, then $B= \{m':\{n':m'n'\in A\}\in p\}\in q.$ Let $a\in B$, and so $a^{-1}A\in p.$ Let $F_1$ be a finite set of polynomials with no constant term and  $\varPhi=\left(f_{n}:n\geq2\right)$ be a given sequence of functions. Let $F_{n,x}$
be a finite collection of polynomials with no constant term and having
the degree and the absolute value of each of the coefficients is less
or equal to $f_{n}\left(x\right).$

Let $G_1\in \mathcal{P}_f(\mathbb{P})$  defined as follows: 
\[g\in G_1 \Leftrightarrow \exists f\in F_1 \ \forall x\in\N \ g(x)=f(ax).\]
Define a new sequence of functions $\varPhi^{\star}=\left(g_{m}:m\geq 2\right)$ by $g_m(x)=f_m(ax)$ for all $x\in \N$. As $a^{-1}A\in  p,$ we have a sequence $\langle x_n \rangle_{n=1}^\infty$, such that 
\[
\mathcal{PF}^n_{k,\varPhi^\star}\left(x_{k}\right)_{k\in\mathbb{N}}=\left\{ x_{k}\cdot n^{\sum_{i=1}^{k-1}q_{i}\left(x_{i}\right)}:k\in\mathbb{N},q_{1}\in G_{1},\text{ and }q_{i}\in G_{i,g_{i}\left(x_{i-1}\right)}\text{ for }2\leq i\leq k-1\right\} \subset a^{-1} A.
\]
In other words,
\[
\mathcal{PF}^n_{k,\varPhi}\left((ax)_{k}\right)_{k\in\mathbb{N}}=\left\{ ax_{k}\cdot n^{\sum_{i=1}^{k-1}p_{i}\left(ax_{i}\right)}:k\in\mathbb{N},p_{1}\in F_{1},\text{ and }p_{i}\in F_{i,f_{i}\left(x_{i-1}\right)}\text{ for }2\leq i\leq k-1\right\} \subset A.
\]
This completes the proof.
\end{proof}

The next lemma needs some results from Topological dynamics.
Let $X$ be a compact space and $(S,\cdot)$ be a semigroup. 
For each $s\in S,$ let $T_s:X\rightarrow X$ be a continuous map and $T_{st}=T_s\circ T_t.$ Then $(X, \langle T_s\rangle_{s\in S})$ is called a dynamical system. A point $x\in X$ is called a uniformly recurrent point if for every open set $U\subseteq X,$ the return times set $\{s:T_sx\in U\}$ is a syndetic set.

\begin{lem}\label{multhick}
    Let $n\in \N_{>1}$, and $L^a$ and $L^m$ be both left ideals of $(\bN,+)$ and $(\bN,\cdot).$ Then there exist $p\in L^a$ and $q\in L^m$ such that $p\star_n q=q.$
\end{lem}
\begin{proof}
     Fix $n\in \N_{>1}$. For each $s\in \N,$ let $T_s:\bN\rightarrow \bN$ be the continuous extension of the map $f_s:\N\rightarrow \N$ given by $f_s(t)=n^st$. Clearly $(\bN,\langle T_s\rangle_{s\in \N})$ is a topological dynamical system. Since $L^m$ is a left ideal of $(\bN,\cdot)$, we have $T_s[L^m]\subseteq L^m$ for all $s\in \N.$ Hence $(L^m,\langle T_s\rangle_{s\in \N})$ is also a topological dynamical system. Let $q$ be a uniformly recurrent point of the dynamical system $(L^m,\langle T_s\rangle_{s\in \N})$. From \cite[Lemma 2.1]{hsz}, there exists $p\in L^a$ such that $T_p(q)=p\star_n q=q.$ Also from \cite[Lemma 2.1]{hsz}, we can choose $p$ as an idempotent of $L^a.$
\end{proof}

The following Theorem shows that every multiplicative piecewise syndetic set contains the conclusion of the Theorem \ref{main2}, which will be later helpful to prove Theorem \ref{final}.
%polynomial Di Nasso-Ragosta configuration.
%\textcolor{red}{please highlight this one in the intro.}
\begin{lem}\label{surprise}
    $\overline{K\left(\bN,\cdot\right)}\subseteq \mathcal{DR}.$ In other words, for every $n(>1)\in \N,$ and $p\in \overline{K\left(\bN,\cdot\right)},$ each member of $n^p$ contains Hindman tower of a sequence (the conclusion of Theorem \ref{important}).
    %witness the conclusion of Theorem \ref{important}.
\end{lem}
\begin{proof}
   Let $p\in \overline{K\left(\bN,\cdot\right)},$ and $A\in p.$ Then $\overline{A}\cap K\left(\bN,\cdot\right)\neq \emptyset.$ Hence it will be sufficient to prove $K\left(\bN,\cdot\right)\subseteq \mathcal{DR}.$

   From \cite[Theorem 1.51]{key-21} we know that $K\left(\bN,\cdot\right)=\bigcup \{L:L\text{ is a minimal left ideal of }(\bN,\cdot)\}.$
    Choose any minimal left ideal $L$ of $(\bN,\cdot).$ From Lemma \ref{multhick}, there exists $p\in L$ such that $p\in \mathcal{DR}.$ Again from Lemma \ref{dilation}, we have $\bN\cdot p\subseteq \mathcal{DR}.$ But $L$ is minimal. Hence $\bN\cdot p=L.$ Hence $L\subseteq \mathcal{DR}.$ As $L$ is an arbitrary minimal left ideal, we have $K\left(\bN,\cdot\right)\subseteq \mathcal{DR}.$
\end{proof}

Now we are in the position to prove Theorem \ref{final}.

\begin{proof}[Proof of Theorem \ref{final}]
The proof is similar to the proof of Theorem \ref{trivial}. 
Choose any $p\in K(\beta\N,\cdot)$ and $A\in p$. Then by \cite[Theorem 4.39]{key-22}, the set $S=\{x\in \N:x^{-1}A\in p\}$ is syndetic. Hence there exists $\{b_1,\ldots ,b_m\}\subset \N$ such that $FS\left(\{b_1,\ldots ,b_m\}\right)\subset S.$ Hence $B=A\cap \bigcap_{x\in FS\left(\{b_1,\ldots ,b_m\}\right)}x^{-1}A\in p.$ Now from Theorem \ref{main2}, there exist $a_1,\ldots ,a_n\in \N$ such that $a_n\cdot 2^{\sum_{i=1}^{N-1}\lambda_ia_i}\in B$ for some suitable choice of $(\lambda_i)_{i=1}^{N-1}.$ Now redefining variables like in the proof of Theorem \ref{trivial} we have the monochromatic solution of the given equation.
\end{proof}

\subsection{Combined Exponential and Product Hindman Theorem}
%\begin{comment}
Before we conclude our article, we also prove two variants (Theorem \ref{hi1}, \ref{hi2}) of the combined additive-multiplicative Hindman theorem originally proved in \cite{h1,key-51}.
%In \cite[Theorem 2.4.]{key-51}, Bergelson and Hindman proved the following combined extension of the additive and multiplicative Hindman theorem. 
\begin{thm}[\cite{h1,key-51}]\label{fsfp} 
    If $\N$ is finitely colored, then there exist two sequences $\langle x_n\rangle_n$ and $\langle y_n\rangle_n$ such that $FS(\langle x_n\rangle_n)\cup FP(\langle y_n\rangle_n)$ is in the same color.
\end{thm}
Let $FEX\left(\langle x_n\rangle_n\right)$ be the Hindman tower generated by the sequence $\langle x_n\rangle_n.$
%For any sequence $\langle x_n\rangle_n,$ denote by $FEX\left(\langle x_n\rangle_n\right)$ to be the Hindman tower generated by the sequence $\langle x_n\rangle_n.$ 
In the exponential counterpart of Theorem \ref{fsfp}, the following theorem of J. Sahasrabudhe is finitary but strong enough that state that the sequences $\langle x_n\rangle_n$ and $\langle y_n\rangle_n$ are the same in Theorem \ref{fsfp}.
\begin{thm}[\cite{key-4}]\label{oohao}
    If $m\in \N,$ and $\N$ is finitely colored, then there exists a sequence $\langle x_n\rangle_n$ such that $FEX\left(\langle x_n\rangle_{n=1}^m\right)\cup FP\left(\langle x_n\rangle_{n=1}^m\right)$ is monochromatic.
\end{thm}

We don't know if Theorem \ref{oohao} is true in an infinitary case, but we have a relatively weaker result which is the exponential version of the Theorem \ref{fsfp}.

\begin{thm}\label{hi1}
    For every finite partition of $\N,$ there exist two sequences $\langle x_n\rangle_n$ and $\langle y_n\rangle_n$ such that $FEX(\langle x_n\rangle_n)\cup FP(\langle y_n\rangle_n)$ is monochromatic.
\end{thm}
\begin{comment}
    \begin{proof}
    Choose $p \in \overline{E\left(K\left(\beta\mathbb{N},+\right)\right)}\cap E\left(K\left(\beta\mathbb{N},\cdot\right)\right)$ and $A\in p.$ As $A$ is additive central, there exists a sequence $\left(a_n\right)_n$ such that $FS\left(\left(a_n\right)_n\right)\subseteq A$. Hence $n^A$ contains patterns of the form $FP(\langle y_n\rangle_n),$ whereas patterns of the form $FEX(\langle x_n\rangle_n)$ follow form Lemma \ref{surprise}.
\end{proof}
\end{comment}

%The following extension of Theorem \ref{fsfp} shows that we can include exponential patterns inside Theorem \ref{fsfp}.
Note that to prove Theorem \ref{hi1}, we don't need the following lemma, in its full strength, which extends Theorem \ref{fsfp}, but we think it is valuable to mention.

\begin{lem}\label{hi2}
     Let $n\in \N,$ $F_{1}\in \mathcal{P}_f(\mathbb{P})$. For every sequence $\varPhi=\left(f_{n}:n\geq2\right)$
of functions $f_{n}:\mathbb{N}\rightarrow\mathbb{N},$ let $F_{n,x}$
be a finite collection of polynomials defined as in Theorem \ref{main2}. Then for every finite partition of $\N$, there exist three sequences $\langle x_k\rangle_{k=1}^\infty, \langle y_k\rangle_{k=1}^\infty, \langle z_k\rangle_{k=1}^\infty $ such that $$\mathcal{PF}^n_{k,\varPhi}\left(x_{k}\right)_{k=1}^\infty\cup FS(\langle y_k\rangle_{k=1}^\infty)\cup FP(\langle z_k\rangle_{k=1}^\infty) $$ is in the same partition.
\end{lem}
\begin{proof}
    Let $p\in E\left( K(\beta\mathbb{N},+)\right)$ and $L=\bN\cdot (p\star_n p).$ Now $L$ is a left ideal of $(\bN,\cdot ).$ So it contains an idempotent say $q\cdot (p\star_n p).$ Let $A\in q\cdot (p\star_n p)=(q\cdot  n^{p})\cdot p.$ Now $q\cdot (p\star_n p)$ is an idempotent, hence there exists a sequence $\langle z_k\rangle_{k=1}^\infty$ such that $FP(\langle z_k\rangle_{k=1}^\infty)\subset A.$ As $p\in E\left(K(\beta\mathbb{N},+)\right) \subset \overline{E\left(K(\beta\mathbb{N},+)\right)}$, and $ \overline{E\left(K(\beta\mathbb{N},+)\right)}$ is a left ideal of $(\bN,\cdot)$, we have $(q\cdot n^{p})\cdot p\in \overline{E\left(K(\beta\mathbb{N},+)\right)}$. So $A$ contains $FS(\langle y_k\rangle_{k=1}^\infty)$ for some sequence $\langle y_k\rangle_{k=1}^\infty.$ And from Lemma \ref{dilation}, we have a sequence $\langle x_n\rangle_{n=1}^\infty$ such that $\mathcal{PF}^n_{k,\varPhi}\left(x_{k}\right)_{k=1}^\infty\subset A.$

    This completes the proof.
\end{proof}

\begin{proof}[Proof of Theorem \ref{hi1}]
Choose any ultrafilter $p\in \bN$ each of whose members contains the combinatorial conclusion of the Lemma \ref{hi2}.  Then each member of $n^p$ will be the desired ultrafilter each of whose members contains the pattern in Theorem \ref{hi1}. This completes the proof.
    
\end{proof}

%\subsection{Combined additive, exponential and multiplicative structure}
%\end{comment}

\section{Thick Sets Not Containing Sahasrabudhe-Schur pattern}\label{count}

From J. Sahasrabudhe's work \cite{key-4}, we know that patterns of the form $\{a,b,a+b,x,y,x^{y}\}$ is not partition regular and hence all ultrafilters in $\overline{K(\bN,\cdot)}$ cannot be witnesses of the partition regularity of exponential triples, as they are already witnesses of the partition regularity of $\{a,b,a+b\}.$
%we cannot do the same by working directly with MPS sets, as the fact that the configuration $\{a,b,a+b,x,y,x^{y}\}$ is not partition regular forces that all ultrafilters in $\overline{K(\bN,\odot)}$ cannot be witnesses of the partition regularity of exponential triples, as they are already witnesses of the partition regularity of $\{a,b,a+b\}$; in particular, there must be infinitely many MPS sets that do not contain exponential triples. 
In this section, we construct an additive and multiplicative thick-set, that does not contain any exponential triple.

%we want to (slightly) improve this result to prove that containing exponential triples and having a rich additive or multiplicative structure are very different properties. In fact, we have the following result:

\begin{thm}\label{ex1}
    There exists a set $A$ which is both additively and multiplicatively thick and does not contain any exponential triple.
\end{thm}
\begin{proof}
    We construct inductively longer and longer intervals $I_{n}=\left[x_{n},y_{n}\right]$ so that $A=\bigcup_{n\in \N} I_{n}$ has the desired properties.
    
    For $n=1$, just take $x_1=y_1>2$. 
    
    Assuming to have defined $I_{n}=\left[x_{n},y_{n}\right]$, pick $x_{n+1},y_{n+1}$ such that
    \begin{enumerate}
        \item $y_{n}^{y_{n}}<x_{n+1}< y_{n+1} < x_{n+1}^{2} < 2^{x_{n+1}}$ and 
        \item $y_n=x_n \cdot y_{n-1}$.
    \end{enumerate}

    Let $A=\bigcup_{n\in \mathbb{N}}[x_n,y_n]$. Then:
    \begin{enumerate}
        \item $A$ is additively thick as condition (2) above entails that $|I_{n}|\rightarrow \infty$ for $n\rightarrow \infty$;
        \item $A$ is multiplicatively thick, as for any $n\in\N$, if we take $y_{p}$ with $y_{p}>n$ we have that $x_{p},2x_{p},\dots,nx_{p}\in I_{p}\subseteq A$;
        \item $A$ does not contain any exponential triple. In fact, let $x,y\in A$. If $x<y$, let $y\in I_{n}$. Then 
    $$y_n<2^{x_n}<x^{x_n}<x^y<y^y\leq y_n^{y_n}<x_{n+1}.$$ 
    Hence $x^y\notin A.$
    
    If $x>y$, analogously let $x\in I_{n}$. Then $$y_n<x_n^2<x_n^y<x^y<y_n^y<y_n^{y_n}<x_{n+1}.$$
    So $x^y\notin A.$
     \end{enumerate}
    This completes the proof.
\end{proof}

\begin{comment}

In Corollary \ref{surprise} we proved that $K\left(\bN,\odot\right)\subseteq \mathcal{DR}.$ The following construction shows that even an additively thick set may not contain Di Nasso-Ragosta configuration.
%In the following theorem we will show that Theorem \ref{thickrich} does not hold for additively thick sets.

\begin{thm}\label{another}
     For any $m\in \N$, there exists an additive thick set that does not contain configuration of the form $\{x,y,x\cdot m^y\}$.
\end{thm}
\begin{proof}
    First choose $x_1=y_1>m$. Now for each $n\geq 2,$ choose $x_n,y_n$ inductively satisfying the following conditions
    \begin{enumerate}
        \item $y_{n-1}\cdot m^{y_{n-1}}<x_n<y_n<x_n\cdot m^2<m^{x_n +1}$ (again $x_n>>y_{n-1}\cdot m^{y_{n-1}}$) and 
        \item $y_n-x_n\rightarrow \infty$.
    \end{enumerate}
    Let $A=\bigcup_{n\in \mathbb{N}}[x_n,y_n].$ Then $A$ is a thick set. Assuming $x,y\in A$,  $x<y$, and $y\in [x_n,y_n],$  we have  
    $$y_n<m^{x_n+1}<m\cdot m^{x_n}<x\cdot m^{x_n}<x\cdot m^y<y\cdot m^y<y_n\cdot m^{y_n}<x_{n+1}.$$ 
    Hence $x\cdot m^y\notin A.$ 
    Again if $x>y$ and $x\in [x_m,y_m],$ we have
    $$y_n<x_n \cdot m^2<x_n\cdot m^y<x\cdot m^y<x\cdot m^x<y_n\cdot m^{y_n}<x_{n+1}.$$
    Hence $x\cdot m^y\notin A.$
    This completes the proof.
\end{proof}
\end{comment}

\section*{Acknowledgement} The authors are thankful to Prof. D. Strauss for her invaluable discussions on the nonexistence of $\star_n$ idempotent ultrafilters that have been addressed in Section \ref{sour}.
The first author of this paper is supported by NBHM postdoctoral fellowship with reference no: 0204/27/(27)/2023/R \& D-II/11927.

%\textcolor{blue}{DO WE HAVE OTHERS?}

\end{document}